\documentclass[11pt]{article}
\usepackage{amsmath,amssymb,amsfonts,amsthm,graphicx}
\usepackage{hyperref}
\newtheorem{theorem}{Theorem}[section]
\newtheorem{proposition}[theorem]{Proposition}
\newtheorem{corollary}[theorem]{Corollary}
\newtheorem{lemma}[theorem]{Lemma}
\newtheorem{remark}[theorem]{Remark}
\newtheorem{definition}[theorem]{Definition}
\numberwithin{equation}{section}
\title{A dependence with complete connections approach to generalized R\'enyi continued fractions}
\author{
    Gabriela Ileana Sebe\footnote{e-mail: igsebe@yahoo.com.} \\
    \emph{\small Politehnica University of Bucharest, Faculty of Applied Sciences},\\
    \emph{\small Splaiul Independentei 313, 060042, Bucharest, Romania} and \\
    \emph{\small Institute of Mathematical Statistics and Applied Mathematics}, \\
     \emph{\small Calea 13 Sept. 13, 050711 Bucharest, Romania} \\
    and\\
    Dan Lascu\footnote{e-mail: lascudan@gmail.com.}\nonumber \\
    \emph{\small Mircea cel Batran Naval Academy, 1 Fulgerului, 900218 Constanta,
    Romania} \\
    }
\begin{document}
\maketitle
\thispagestyle{empty}
\begin{abstract}
We introduce and study in detail a special class of backward continued fractions that represents a generalization of R\'enyi continued fractions.
We investigate the main metrical properties of the digits occurring in these expansions and we construct the natural extension for the transformation that generates the R\'enyi-type expansion.
Also we define the associated random system with complete connections whose ergodic behavior allows us to prove a variant of Gauss-Kuzmin-type theorem.
\end{abstract}

\section{Introduction}

In \cite {Grochenig&Haas1-1996}, Gr\"ochenig and Haas investigated $u$-backward continued fractions associated with the one parameter family of interval maps of the form $T_u (x) :=  \frac{1}{u(1-x)} - \lfloor \frac{1}{u(1-x)} \rfloor$,
where $u>0$, $x \in [0, 1)$ and $\lfloor \cdot \rfloor$ denotes the floor function.
As the parameter varies the maps $T_u$ exhibit a curious dynamical behavior.
It was shown that varying $u$ in the interval $(0,4)$ there is a viable theory of $u$-backward continued fractions, which fails when $u \geq 4$.
So, for a given $u \in (0, 4)$, each irrational $x \in [0, 1)$ can be uniquely represented as an infinite continued fraction
\begin{equation}
x = 1 - \displaystyle \frac{1}{u n_1 - \displaystyle \frac{1}{n_2 - \displaystyle \frac{1}{u n_3 - \ddots - \displaystyle \frac{1}{s_k n_k - \ddots} }}} :=[a_1, a_2, a_3, \ldots]_u, \label{1.1}
\end{equation}
where the integers $n_i=1+a_i \geq 2$ and the ``coefficient'' $s_j$ of $n_j$ is alternating between $s_j = 1$ for even $j$ and $s_j = u$ for odd $j$.
The case $u = 1$ was studied by R\'enyi \cite{Renyi-1957} and provides an alternate approach to continued fractions and rational approximation.
The graph of $T_1$ can be obtained from the graph of the regular (Gauss) continued fraction transformation $G(x):=\frac{1}{x} - \lfloor \frac{1}{x} \rfloor$ by reflecting about the line $x=\frac{1}{2}$.
%
%
%
%
%
%

The main purpose of Gr\"ochenig and Haas was to find an explicit form of an absolutely continuous invariant measure for $T_u$ similar to the Gauss
measure $\frac{dx}{x+1}$ for $G$ and R\'enyi's measure $\frac{dx}{x}$ for $T_1$.
While the Gauss measure is finite, the R\'enyi's measure is infinite.
They showed that the invariant measure for $T_u$ is finite if and only if $0 < u < 4$ and $u_q \neq 4\cos^2 \frac{\pi}{q}$, $q=3, 4, \ldots$.
Also, they have identified the invariant probability measure for $T_u$ corresponding to the values $u=1/N$ for positive integers $N \geq 2$.
This is a finite measure.
In this particular case we will call the continued fraction in (\ref{1.1}) \textit{R\'enyi-type continued fraction},
we will note $T_{\frac{1}{N}}$ with $R_N$ and we will call it \textit{R\'enyi-type (continued fraction) transformation}.

Several related families of transformations have been investigated with attention to their dynamical properties as well as to the associated continued fraction theories \cite{Burton&Kraaikamp-2000, Haas-2002, Haas&Molnar-2004}.
This has lead to a wide range of connections between continued fractions, diophantine approximation, ergodic theory and hyperbolic geometry \cite{Adler&Flatto-1991, Bedford-1991, Series-1985}.

Our goal in this paper is to start an approach to the metrical theory of R\'enyi-type continued fraction expansions via dependence with complete connections.
Using the natural extensions for the R\'enyi-type transformations, we give an infinite-order-chain representation of the sequence of the incomplete quotients of these expansions.
Then we show that the associated random systems with complete connections are with contraction and their transition operators are regular w.r.t. the Banach space of Lipschitz functions. This leads further, in Section 6, to a solution of Gauss-Kuzmin problem for these continued fractions.
\section{R\'enyi-type continued fraction expansions as dynamical system} \label{section2}
For a fixed integer $N \geq 2$, we define R\'enyi-type continued fraction transformation $R_N: [0,1] \rightarrow [0,1]$ by
\begin{equation}
R_{N}(x) :=
\left\{
\begin{array}{ll}
{\displaystyle \frac{N}{1-x}- \left\lfloor\frac{N}{1-x}\right\rfloor},&
{ x \in [0, 1) }\\
0,& x=1.
\end{array}
\right. \label{2.1}
\end{equation}
To this transformation, we associate the digits $a_n(x)$, $n \in {\mathbb{N}}_+:=\{1, 2,3, \ldots \}$, which are defined by
\begin{equation}
a_n := a_n(x) = a_1\left( R^{n-1}_N (x) \right), \quad n \geq 2, \label{2.2}
\end{equation}
with $R_{N}^0 (x) = x$ and
\begin{equation}
a_1:=a_1(x) = \left\{\begin{array}{lll}
\left\lfloor \frac{N}{1-x} \right\rfloor & \hbox{if} & x \neq 1, \\
\infty & \hbox{if} & x = 1.
\end{array} \right. \label{2.2'}
\end{equation}
Putting $\Lambda:=\{N, N+1, \ldots\}$, observe that $a_n \in \Lambda$ for any  $n \in \mathbb{N}_+$.
Using these digits we can write
\begin{equation}
R_{N}(x) := \frac{N}{1-x} - a_1(x). \label{2.3}
\end{equation}
Then
\begin{equation} \label{2.4}
x = 1 - \displaystyle \frac{N}{1+a_1 - \displaystyle \frac{N}{1+a_2 - \displaystyle \frac{N}{1+a_3 - \ddots}}} :=[a_1, a_2, a_3, \ldots]_R.
\end{equation}
The R\'enyi-type continued fraction in (\ref{2.4}) is treated as a dynamical system $\left([0,1],{\mathcal B}_{[0,1]}, R_N, \rho_N \right)$,
where $\mathcal{B}_{[0,1]}$ denotes the $\sigma$-algebra of all Borel subsets of $[0,1]$,
and
\begin{equation}
\rho_N (A) :=
\frac{1}{\log \left(\frac{N}{N-1}\right)} \int_{A} \frac{\mathrm{d}x}{x+N-1}, \quad A \in {\mathcal{B}}_{[0,1]} \label{2.5}
\end{equation}
is the invariant probability measure under $R_N$ \cite {Grochenig&Haas1-1996}.

Let $\frac{p_n}{q_n} := \frac{p_n(x)}{q_n(x)}$ denote the partial fractions obtained by applying the transformation
$R_N$ $n$-times, so
\begin{equation} \label{2.6}
\frac{p_n}{q_n} = 1 - \displaystyle \frac{N}{1+a_1 - \displaystyle \frac{N}{1+a_2 - \ddots - \displaystyle \frac{N}{1+a_n}}}.
\end{equation}
For the functions $p_n$ and $q_n$ we obtain the following recurrence relations:
\begin{eqnarray}
p_0&=&1, \ p_1=1+a_1-N, \quad p_n = (1+a_n) p_{n-1} - N p_{n-2}, n \geq 2, \label{2.7} \\
q_0&=&1, \ q_1=1+a_1,  \ \ \ \ \quad \quad q_n = (1+a_n) q_{n-1} - N q_{n-2}, n \geq 2. \label{2.8}
\end{eqnarray}
Using these recurrences, induction easily gives that
\begin{equation}
p_{n-1}q_n - p_nq_{n-1} = N^n, \quad n \in \mathbb{N_+} \label{2.9}
\end{equation}
and
\begin{equation}
x = \frac{p_n + (R^n_N - 1)p_{n-1}}
{q_n + (R^n_N - 1)q_{n-1}}, \quad n \in \mathbb{N_+}. \label{2.10}
\end{equation}
We obtain
\begin{equation}
\left| x - \frac{p_n}{q_n} \right| \leq \frac{N^n}{q_n(q_n - q_{n-1})}, \quad n \in \mathbb{N}_+. \label{2.11}
\end{equation}

\section{The probabilistic structure of $(a_n)_{n \in \mathbb{N_+}}$ under the Lebesgue measure} \label{section3}

We start by defining the $n$-th order cylinder associated to the digits $(a_n)_{n \in \mathbb{N_+}}$ of the R\'enyi-type continued fraction (\ref{2.3}).
An $n$-block $(a_1, a_2, \ldots, a_n)$ is said to be \textit{admissible} for the expansion in (\ref{2.3}) if there exists $x \in [0, 1)$ such that $a_i(x)=a_i$ for all $1 \leq i \leq n$.
If $(a_1, a_2, \ldots, a_n)$ is an admissible sequence, we call the set
\begin{equation}
I (a_1, a_2, \ldots, a_n) = \{x \in [0, 1]:  a_1(x) = a_1, a_2(x) = a_2, \ldots, a_n(x) = a_n \}, \label{3.1}
\end{equation}
\textit{the $n$-th order cylinder}. As we mention above, $(a_1, a_2, \ldots, a_n) \in \Lambda^n$.
For example, for any $a_1=i \in \Lambda$ we have
\begin{equation}
I\left( a_1\right) = \left\{x \in [0, 1]: a_1(x) = a_1 \right\} = \left[ 1 - \frac{N}{i}, 1 - \frac{N}{i+1} \right). \label{3.2}
\end{equation}
By induction, it can be shown that
\begin{equation}
I (a_1, a_2, \ldots, a_n) = \left[ \frac{p_n-p_{n-1}}{q_n-q_{n-1}}, \frac{p_n}{q_n} \right), \label{3.3}
\end{equation}
for all $ n \geq 1$. Hence, using (\ref{2.9})
\begin{equation} \label{3.4}
\lambda \left(I (a_1, a_2, \ldots, a_n) \right) = \frac{N^n}{q_n (q_n - q_{n-1})},
\end{equation}
where $\lambda$ is the Lebesgue measure and $(q_n)$ is as in (\ref{2.8}).

To derive the so-called Brod\'en-Borel-L\'evy formula \cite{IG-2009, IK-2002},
let us define $(s_n)_{n \in {\mathbb{N}}}$ by
\begin{equation}\label{3.5}
  s_0:=1, \quad s_n := 1- N\frac{q_{n-1}}{q_n}.
\end{equation}
From (\ref{2.8}), $s_n=1-N/(a_n + s_{n-1})$.
Hence,
\begin{equation}\label{3.6}
 s_n = 1 - \displaystyle \frac{N}{1+a_n - \displaystyle \frac{N}{1+a_{n-1} - \ddots - \displaystyle \frac{N}{1+a_1 }}} =
  [a_n, a_{n-1}, \ldots, a_1]_R.
\end{equation}
\begin{proposition} [Brod\'en-Borel-L\'evy-type formula] \label{prop.BBL}
Let $\lambda$ denote the Lebesgue measure on $[0, 1]$.
For any $n \in \mathbb{N}_+$, the conditional probability
$\lambda (R^n_N < x |a_1, \ldots, a_n )$ is given as follows:
\begin{equation}
\lambda (R^n_N < x |a_1,\ldots, a_n )
= \frac{Nx}{N-(1-x)(1-s_n)}, \quad x \in [0, 1] \label{3.7}
\end{equation}
where $(s_n)$ is as in $(\ref{3.5})$ and $a_1, \ldots, a_n$ are as in $(\ref{2.2})$ and $(\ref{2.2'})$.
\end{proposition}
\begin{proof}
By definition, we have
\begin{equation} \nonumber
\lambda\left(R^n_N < x |a_1,\ldots, a_n \right) =
\frac{\lambda\left(\left(R^n_N < x\right) \cap I(a_1,\ldots, a_n) \right)}{\lambda\left(I(a_1,\ldots, a_n)\right)}
\end{equation}
for any $n \in \mathbb{N}_+$ and $x \in [0, 1]$.
From (\ref{2.10}) and (\ref{3.3}) we have
\begin{eqnarray} \nonumber
\lambda\left(\left(R^n_N < x \right) \cap I(a_1,\ldots, a_n)\right) &=&
\left|\frac{p_n+(x-1)p_{n-1}}{q_n+(x-1)q_{n-1}} - \frac{p_n-p_{n-1}}{q_n-q_{n-1}}\right|  \\
\nonumber \\
&=& \frac{N^nx}{(q_n-q_{n-1})(q_n+(x-1)q_{n-1})}. \nonumber
\end{eqnarray}
From this and (\ref{3.4}), we have
\begin{eqnarray}
\lambda\left(R^n_N < x |a_1, \ldots, a_n \right) &=&
\frac{\lambda\left(\left(R^n_N < x\right) \cap I(a_1,\ldots, a_n) \right)}{\lambda\left(I(a_1,\ldots, a_n)\right)} \nonumber \\
\nonumber \\
&=& \frac{xq_n}{q_n+(x-1)q_{n-1}} = \frac{Nx}{N - (1-x)(1-s_n)} \nonumber
\end{eqnarray}
for any $n \in \mathbb{N}_+$ and $x \in [0, 1]$.
\end{proof}
The Brod\'en-Borel-L\'evy-type formula allows us to determine the probabilistic structure of the digits $(a_n)_{n \in \mathbb{N}_+}$ under $\lambda$.
\begin{proposition} \label{prop.BBLConsecinta}
For any $i \geq N$ and $n \in \mathbb{N}_+$, we have
\begin{equation}
\lambda(a_1=i) = \frac{N}{i(i+1)}, \quad
\lambda\left(a_{n+1}=i |a_1,\ldots, a_n \right) = P_{N,i}(s_n) \label{3.8}
\end{equation}
where $(s_n)$ is as in (\ref{3.5}), and
\begin{equation}
P_{N,i}(x) := \frac{x+N-1}{(x+i)\,(x+i-1)}. \label{3.9}
\end{equation}
\end{proposition}
\begin{proof}
From (\ref{3.2}), the case $\lambda(a_1=i)$ holds.
For $n \in \mathbb{N}_+$ and $x \in [0, 1]$, we have
$R_N^n(x) = [a_{n+1}, a_{n+2}, \ldots]_R$,
where $(a_n)$ is as in (\ref{2.2}).
By using (\ref{3.7}), we have
\begin{eqnarray}
\lambda(\,a_{n+1}=i \,|\,a_1,\ldots, a_n\, )
& = & \lambda\left( \,R^n_N \in \left[1-\frac{N}{i}, 1-\frac{N}{i+1} \right)\,
| \,a_1,\ldots, a_n \,\right) \nonumber \\
&=& \frac{N \left(1-\frac{N}{i+1}\right)}{N-\frac{N}{i+1}(1-s_n)} - \frac{N \left(1-\frac{N}{i}\right)}{N-\frac{N}{i}(1-s_n)}\nonumber \\
\nonumber\\
&=& P_{N,i}(s_n). \nonumber
\end{eqnarray}
\end{proof}

It is easy to check that
\begin{equation}
\sum_{i=N}^{\infty}P_{N,i}(x) = 1 \quad \mbox{ for any } x \in [0, 1]. \label{3.10}
\end{equation}
\begin{remark}
Proposition \ref{prop.BBLConsecinta} is the starting point of an approach to the metrical
theory of R\'enyi-type continued fraction expansions via dependence with complete connections
(see \cite{IG-2009}, Section 5.2).
We apply this method in Section 6 to obtain a solution of Gauss-Kuzmin-type problem for R\'enyi-type continued fraction expansions.
\end{remark}
\begin{corollary}
The sequence $(s_n)_{n \in \mathbb{N}_+}$ with $s_0 = 1$ is a homogeneous $[0, 1]$-valued Markov chain on
$\left([0, 1], \mathcal{B}_{[0, 1]}, \lambda \right)$ with the following transition mechanism:
from state $s \in [0, 1]$ the only possible one-step transitions are those to states $1-N/(s+i)$, $i \geq N$, with corresponding  probabilities $P_{N,i}(s)$, $i \geq N$.
\end{corollary}

\section{Natural extension and extended random variables}

Fix an integer $N \geq 2$.
In this section, we introduce the natural extension $\overline{R}_N$ of $R_N$ in (\ref{2.1}) and its extended random variables \cite{IK-2002}.
\subsection{Natural extension}

Let $\left([0, 1],{\mathcal B}_{[0, 1]}, R_N \right)$ be as in Section 2.
Define $(u_{N,i})_{i \geq N}$ by
\begin{equation}
u_{N,i}: [0, 1] \rightarrow [0, 1]; \quad
u_{N,i}(x) := 1 - \frac{N}{x+i}, \quad x \in [0, 1]. \label{4.1}
\end{equation}
\noindent
For each $i \geq N$, $u_{N,i}$ is a right inverse of $R_N$, that is,
\begin{equation} \label{4.2}
\left(R_N \circ u_{N,i}\right)(x) = x, \quad \mbox{for any } x \in [0, 1].
\end{equation}
Furthermore, if $a_1(x)=i$, then $\left(u_{N,i} \circ R_N \right)(x)=x$ where $a_1$ is as in (\ref{2.2'}).
\begin{definition} \label{def.natext}
The natural extension $\left([0, 1]^2, {\mathcal B}_{[0, 1]^2},\overline{R}_N \right)$ of $\left([0, 1],{\mathcal B}_{[0, 1]}, R_N \right)$ is the transformation $\overline{R}_N$ of the square space
$\left([0, 1]^2,{\mathcal B}_{[0, 1]}^2 \right):=\left([0, 1], {\mathcal B}_{[0, 1]}\right) \times \left([0, 1], {\mathcal B}_{[0, 1]}\right)$
defined as follows \cite{Nakada-1981}:
\begin{eqnarray} \label{4.3}
&&\overline{R}_N: [0, 1]^2 \rightarrow [0, 1]^2; \nonumber \\
&&\overline{R}_N(x,y) := \left( R_N(x), \,u_{N,a_1(x)}(y) \right), \quad (x, y) \in [0, 1]^2.
\end{eqnarray}
\end{definition}

From (\ref{4.2}), we see that $\overline{R}_N$ is bijective on $[0, 1]^2$ with the inverse
\begin{equation} \label{4.4}
(\overline{R}_N)^{-1}(x, y)
= (u_{N, a_1(y)}(x), \,
R_N(y)), \quad (x, y) \in [0, 1]^2.
\end{equation}
Iterations of (\ref{4.3}) and (\ref{4.4}) are given as follows for each $n \geq 2$:
\begin{eqnarray}
\left(\overline{R}_N\right)^n(x, y) =
\left(\,R^n_N(x), \,[a_n(x), a_{n-1}(x), \ldots, a_2(x),\, a_1(x)+ y - 1 ]_R \,\right), \label{4.5} \\
\nonumber
\\
\left(\overline{R}_N\right)^{-n}(x, y) =
\left(\,[a_n(y), a_{n-1}(y), \ldots, a_2(y), \,a_1(y)+x-1]_R,\, R^{n}_N(y) \,\right).
\label{4.6}
\end{eqnarray}
For $\rho_N$ in (\ref{2.5}), we define its \textit{extended measure} $\overline{\rho}_N$ on $\left([0, 1]^2, {\mathcal{B}}^2_{[0, 1]}\right)$ as
\begin{equation} \label{4.7}
\overline{\rho}_N(B) :=\frac{1}{ \log \left(\frac{N}{N-1}\right) } \int\!\!\!\int_{B}
\frac{N\mathrm{d}x\mathrm{d}y}{\left\{ N-(1-x)(1-y) \right\}^2}, \quad B \in {\mathcal{B}}^2_{[0, 1]}.
\end{equation}
Then
$\overline{\rho}_N(A \times [0, 1]) = \overline{\rho}_N([0, 1] \times A) = \rho_N(A)$ for any $A \in {\mathcal{B}}_{[0, 1]}$.
The measure $\overline{\rho}_N$ is preserved by $\overline{R}_N$, i.e.,
$\overline{\rho}_N ((\overline{R}_N)^{-1}(B))
= \overline{\rho}_N (B)$ for any $B \in {\mathcal{B}}^2_{[0, 1]}$.
\subsection{Extended random variables}

Define the projection $E:[0, 1]^2 \rightarrow [0, 1]$ by $E(x,y):=x$.
With respect to $\overline{R}_N$ in (\ref{4.3}),
define \textit{extended incomplete quotients} $\overline{a}_l(x,y)$,
$l \in \mathbb{Z}:=\{\ldots, -2, -1, 0, 1, 2, \ldots\}$ at $(x, y) \in [0, 1]^2$ by
\begin{equation} \label{4.8}
\overline{a}_{l}(x, y) := (a_1 \circ E)\left(\,(\overline{R}_N)^{l-1} (x, y) \,\right),
\quad l \in \mathbb{Z}.
\end{equation}
\begin{remark} \label{rem.4.2}

\begin{enumerate}
\item[(i)]
Remark that $\overline{a}_{l}(x, y)$ in (\ref{4.6}) is also well-defined for $l \leq 0$ because $\overline{R}_N$ is invertible.
By (\ref{4.3}) and (\ref{4.4}) we have
\begin{equation} \label{4.9}
\overline{a}_n(x, y) = a_n(x), \quad
\overline{a}_0(x, y) = a_1(y), \quad
\overline{a}_{-n}(x, y) = a_{n+1}(y),
\end{equation}
for any $n \in \mathbb{N}_+$ and $(x, y) \in [0, 1]^2$.
%
\item[(ii)]
Since the measure $\overline{\rho}_N$ is preserved by $\overline{R}_N$, the doubly infinite sequence $(\overline{a}_l(x,y))_{l \in \mathbb{Z}}$
is strictly stationary (i.e., its distribution is invariant under a shift of the indices) under $\overline{\rho}_N$.
\end{enumerate}
\end{remark}
\begin{theorem} \label{th.4.3}
Fix $(x,y) \in [0, 1]^2$ and let $\overline{a}_{l}:=\overline{a}_l(x,y)$ for $l \in {\mathbb Z}$.
Define $a:= [\overline{a}_0, \overline{a}_{-1}, \ldots]_R$.
Then the following holds for any $x \in [0, 1]$:
\begin{equation} \label{4.10}
\overline{\rho}_N \left( [0, x] \times [0, 1] \,|
\,\overline{a}_0, \overline{a}_{-1}, \ldots \right)
= \frac{Nx}{N - (1-x)(1-a)} \quad \overline{\rho}_N \mbox{-}\mathrm{a.s.}
\end{equation}
\end{theorem}
\begin{proof}
Recall the cylinder in (\ref{3.1}).
Let $I_{n}$ denote the cylinder $I(\overline{a}_0, \overline{a}_{-1}, \ldots, \overline{a}_{-n})$ for $n \in \mathbb{N}$.
We have
\begin{equation} \label{4.11}
\overline{\rho}_N \left( [0, x] \times [0, 1] \left. \right| \overline{a}_0, \overline{a}_{-1}, \ldots \right) =
\lim_{n \rightarrow \infty} \overline{\rho}_N \left( [0, x] \times [0, 1] \left. \right| \overline{a}_0, \ldots, \overline{a}_{-n} \right) \quad \overline{\rho}_N \mbox{-a.s.}
\end{equation}
and
\begin{eqnarray} \label{4.12}
&&\overline{\rho}_N ( [0, x] \times [0, 1] \left. \right| \overline{a}_0, \ldots, \overline{a}_{-n} )
=
\displaystyle{\frac{\overline{\rho}_N ([0, x] \times I_{n})}{\overline{\rho}_N ([0, 1) \times I_{n})}} \nonumber \\
\nonumber \\
 &&= \frac{1}{\log \left(\frac{N}{N-1}\right)}\displaystyle{\frac{1}{\rho_N(I_{n})}\int_{I_{n}}\mathrm{d}y\displaystyle\int^x_0{\frac{N\mathrm{d}u}{\left\{N-(1-u)(1-y)\right\}^2}}} \nonumber \\
 \nonumber \\
 &&= \displaystyle{\frac{1}{\rho_N(I_{n})} \int_{I_{n}} \frac{Nx}{\left\{N-(1-x)(1-y)\right\}^2}\, \rho_N(\mathrm{d}y)} \nonumber \\
 \nonumber \\
 &&= \displaystyle \frac{Nx}{\left\{N-(1-x)(1-y_n)\right\}^2}
\end{eqnarray}
for some $y_n \in I_{n}$. Since
\begin{equation} \label{4.13}
\lim_{n \rightarrow \infty} y_n =[\overline{a}_0, \overline{a}_{-1}, \ldots]_R = a,
\end{equation}
the proof is completed.
\end{proof}

The stochastic property of $(\overline{a}_l)_{l \in \mathbb{Z}}$ under $\overline{\rho}_N$ is given as follows.
\begin{corollary} \label{cor.4.4}
For any $i \geq N$, we have
\begin{equation} \label{4.14}
\overline{\rho}_N (\left.\overline{a}_1 = i\right| \overline{a}_0, \overline{a}_{-1}, \ldots) = P_{N,i}(a) \quad \overline{\rho}_N \mbox{-}\mathrm{a.s.}
\end{equation}
where $a = [\overline{a}_0, \overline{a}_{-1}, \ldots]_R$ and $P_{N,i}$ is as in $(\ref{3.9})$.
\end{corollary}
\begin{proof}
Let $I_{n}$ be as in the proof of Theorem \ref{th.4.3}.
We have
\begin{equation} \label{4.15}
\overline{\rho}_N (\left.\overline{a}_1 = i\,\right| \,
\overline{a}_0, \overline{a}_{-1}, \ldots) = \lim_{n \rightarrow \infty}
\overline{\rho}_N (\left.\overline{a}_1 = i\,\right| \,I_{n}).
\end{equation}
We have
\begin{equation}
(\overline{a}_1 = i) = \left[1-\frac{N}{i}, 1-\frac{N}{i+1}\right) \times [0, 1], \quad i \geq N.
\end{equation}
Now
\begin{eqnarray} \label{4.16}
\overline{\rho_N} \left( \left.\left[1-\frac{N}{i}, 1-\frac{N}{i+1}\right) \times [0, 1] \right| I_{n}\right) &=&
\frac{\overline{\rho}_N \left( \left[1-\frac{N}{i}, 1-\frac{N}{i+1}\right) \times I_{n}\right)}
{\overline{\rho}_N ([0, 1] \times I_{n})} \nonumber \\
\nonumber \\
& = & \frac{1}{\rho_N (I_{n})} \int_{I_{n}} P_{N,i}(y)\, \rho_N(\mathrm{d}y) \nonumber \\
\nonumber \\
& = & P_{N,i}(y_n)
\end{eqnarray}
for some $y_n \in I_{n}$. From (\ref{4.13}), the proof is completed.
\end{proof}

\begin{remark} \label{rem.4.5}
The strict stationarity of $\left(\overline{a}_l\right)_{l \in \mathbb{Z}}$, under $\overline{\rho}_N$ implies that
\begin{equation} \label{4.17}
\overline{\rho}_N(\left.\overline{a}_{l+1} = i\, \right|\, \overline{a}_l,
\overline{a}_{l-1}, \ldots)
= P_{N,i}(a) \quad \overline{\rho}_N \mbox{-}\mathrm{a.s.}
\end{equation}
for any $i \geq N$ and $l \in \mathbb{Z}$, where
$a = [\overline{a}_l, \overline{a}_{l-1}, \ldots]_R$.
The last equation emphasizes that $\left(\overline{a}_l\right)_{l \in \mathbb{Z}}$ is a chain of infinite order in the theory of dependence with complete connections \cite{IG-2009}.
\end{remark}
Define extended random variables $\left(\overline{s}_l\right)_{l \in \mathbb{Z}}$ as
$\overline{s}_l := [\overline{a}_l, \overline{a}_{l-1}, \ldots]_R$, $l \in \mathbb{Z}$.
Clearly, $\overline{s}_l = \overline{s}_{0} \circ (\overline{R}_N)^l$, $l \in \mathbb{Z}$.
It follows from Corollary \ref{cor.4.4} that $\left(\overline{s}_l\right)_{l \in \mathbb{Z}}$ is a strictly stationary $[0, 1)$-valued Markov process on $\left([0, 1]^2,{\mathcal{B}}^2_{[0, 1]}, \overline{\rho}_{N} \right)$
with the following transition mechanism: from state $\overline{s} \in [0, 1]$ the possible transitions are to any state $1 - N/(\overline{s} + i)$ with corresponding transition probability $P_{N,i}(\overline{s})$, $i \geq N$.
Clearly, for any $l \in \mathbb{Z}$ we have
\begin{equation}
\overline{\rho}_{N}(\overline{s}_l < x ) = \overline{\rho}_{N}([0, 1] \times [0,x)) = \rho_{N}([0,x)), \quad x \in [0, 1].
\end{equation}
Motivated by Theorem \ref{th.4.3}, we shall consider the one-parameter family $\{\rho_{N, t}: t \in [0, 1]\}$
of (conditional) probability measures on $\left([0, 1], {\mathcal{B}}_{[0, 1]} \right)$
defined by their distribution functions
\begin{equation}
\rho_{N,t} ([0, x]) := \frac{Nx}{N - (1-x)(1-t)}, \quad x, t \in [0, 1]. \label{4.20}
\end{equation}
Note that $\rho_{N,1} = \lambda$.

For any $t \in [0, 1]$ put
\begin{equation}
s_{0,t} := t,\quad
s_{n,t} := 1 - \frac{N}{a_n + s_{n-1,t}}, \quad n \in \mathbb{N}_+. \label{4.21}
\end{equation}
\begin{remark}
It follows from the properties just described of the process $(\overline{s}_l)_{l \in \mathbb{Z}}$
that the sequence $(s_{n,t})_{n \in {\mathbb{N}}_+}$ is an $[0, 1]$-valued Markov chain on
$\left([0, 1],{\mathcal B}_{[0,1]}, \rho_{N,t} \right)$
which starts at $s_{0,t} := t$ and has the following transition mechanism:
from state $s \in [0, 1]$ the possible transitions are to any state $1 - N/(s+i)$ with corresponding transition probability $P_i(s)$, $i \geq N$.
\end{remark}
\section{Perron-Frobenius operators} \label{section5}

Let $\left([0, 1],{\mathcal B}_{[0, 1]}, R_N, \rho_N \right)$ be as in Section 2.
In this section, we derive its Perron-Frobenius operator.

Let $\mu$ be a probability measure on $\left([0, 1], {\mathcal{B}}_{[0, 1]}\right)$
such that $\mu\left(R_N^{-1}(A)\right) = 0$ whenever $\mu(A) = 0$ for
$A \in {\mathcal{B}}_{[0, 1]}$.
For example, this condition is satisfied if $R_N$ is $\mu$-preserving, that is,
$\mu R_N^{-1} = \mu$.
Let %
$
L^1([0, 1], \mu):=\{f: [0, 1] \rightarrow \mathbb{C} : \int^{1}_{0} |f |\mathrm{d}\mu < \infty \}.
$
The \textit{Perron-Frobenius operator} of $\left([0, 1], {\mathcal B}_{[0, 1]}, R_N, \mu \right)$
can be defined by the Radon-Nikodym theorem as the unique linear and positive operator on the Banach space $L^1([0, 1],\mu)$ satisfying:
\begin{equation}
\int_{A} Uf \,\mathrm{d}\mu = \int_{R_{N}^{-1}(A)}f\, \mathrm{d}\mu \quad
\mbox{ for all }
A \in {\mathcal{B}}_{[0, 1]},\, f \in L^1 \left([0, 1], \mu \right). \label{5.1}
\end{equation}
About more details, see \cite{BG-1997, IK-2002}.
\begin{proposition} \label{prop.5.1}
Let $\left([0, 1],{\mathcal B}_{[0, 1]}, R_{N}, \rho_{N} \right)$ be as in Section 2, and let $U$ denote its Perron-Frobenius operator.
Then the following holds:
\begin{enumerate}
\item[(i)]
The following equation holds:
\begin{equation}
Uf(x) = \sum_{i \geq N} P_{N,i}(x)\,f\left(u_{N,i}(x)\right), \quad
f \in L^1([0, 1], \rho_{N}), \label{5.2}
\end{equation}
where $P_{N,i}$ and $u_{N,i}$ are as in (\ref{3.9}) and (\ref{4.1}), respectively.
\item[(ii)]
Let $\mu$ be a probability measure on $\left([0, 1],{\mathcal{B}}_{[0, 1]}\right)$
such that $\mu$ is absolutely continuous with respect to
the Lebesgue measure $\lambda$
and let $h := \mathrm{d}\mu / \mathrm{d} \lambda$  a.e. in $[0, 1]$.
Then for any $n \in \mathbb{N}_+$ and $A \in {\mathcal{B}}_{[0, 1]}$,
we have
\begin{equation}
\mu \left(R_{N}^{-n}(A)\right)
= \int_{A} U^nf(x) \mathrm{d}\rho_{N}(x) \label{5.3}
\end{equation}
where $f(x):= \left(\log\left(\frac{N}{N-1}\right)\right) (x+N-1) h(x)$ for
$x \in [0, 1]$.
\end{enumerate}
\end{proposition}
\begin{proof}
(i) Let $R_{N,i}$ denote the restriction of $R_N$ to the subinterval
$I(i):=\left(1-\frac{N}{i}, 1-\frac{N}{i+1}\right]$, $i \geq N$, that is,
\begin{equation}
R_{N,i}(x) = \frac{N}{1-x} - i, \quad x \in I(i). \label{5.4}
\end{equation}
Let $C(A):=\left(R_{N}\right)(A)$ and $C_{i}(A):=\left(R_{N,i}\right)^{-1}(A)$ for
$A \in {\mathcal B}_{[0, 1]}$.
Since $C(A)=\bigcup_{i}C_i(A)$ and $C_i\cap C_j$ is a null set when $i \neq j$,
we have
\begin{equation}
\int_{C(A)} f \,\mathrm{d} \rho_N = \sum_{i \geq N} \int_{C_i(A)}f\, \mathrm{d} \rho_N,
\quad
f \in L^1([0, 1], \rho_N),\,A \in {\mathcal{B}}_{[0, 1]}. \label{5.5}
\end{equation}
For any $i \geq N$, by the change of variable
$x = \left(R_{N,i}\right)^{-1}(y) = 1 - \displaystyle \frac{N}{y+i}$,
we successively obtain
\begin{eqnarray}
\int_{C_i(A)}f(x) \,\rho_N(\mathrm{d}x) &=& \left(\log\left(\frac{N+1}{N}\right)\right)^{-1} \int_{C_i(A)} \frac{f(x)}{x+N-1}\,\mathrm{d}x \nonumber \\
\nonumber \\
&=& \left(\log\left(\frac{N+1}{N}\right)\right)^{-1} \int_{A} \frac{1}{(y+i)(y+i-1)}f\left(u_{N,i}(y)\right) \mathrm{d}y \nonumber \\
\nonumber \\
&=& \int_{A} P_{N,i}(y)\, f\left(u_{N,i}(y)\right)\,\rho_N (\mathrm{d}y). \label{5.6}
\end{eqnarray}
Now, (\ref{5.2}) follows from (\ref{5.5}) and (\ref{5.6}).

\noindent
(ii) We will use mathematical induction.
For $n=0$, the equation (\ref{5.3}) holds by definitions of $f$ and $h$.
Assume that (\ref{5.3}) holds for some $n \in \mathbb{N}$.
Then
\begin{equation} \label{5.7}
\mu \left(R_N^{-(n+1)}(A)\right) =
\mu \left(R_N^{-n}\left(R_N^{-1}(A)\right)\right)
= \int_{C(A)} U^n f(x)\,\rho_N(\mathrm{d}x),
\end{equation}
and by definition, we have
\begin{equation} \label{5.8}
\int_{C(A)} U^n f(x) \,\rho_N(\mathrm{d}x) = \int_{A} U^{n+1} f(x) \,\rho_N(\mathrm{d}x).
\end{equation}
Therefore,
\begin{equation} \label{5.9}
\mu \left(R_N^{-(n+1)}(A)\right)
= \int_{A} U^{n+1} f(x)\rho_N(\mathrm{d}x)
\end{equation}
which ends the proof.

\end{proof}

\section{Random systems with complete connections and the Gauss-Kuzmin-type problem}

For everything concerning dependence with complete connections and, in particular, random systems with complete connections (RSCC) we refer the reader to \cite{IG-2009}.
For applications of RSCC in continued fraction expansions see also \cite{Iosifescu&Sebe-2013, Lascu-2013, Sebe-2001, Sebe-2002, Sebe&Lascu-2014}.

In what follows we give the RSCC associated with the R\'enyi-type continued fraction expansions.
The ergodic behaviour of this RSCC gives us the asymptotic of $R^{-n}_N$ as $n \rightarrow \infty$ that represents the Gauss-Kuzmin-type problem for R\'enyi-type continued fraction expansions.

Proposition \ref{prop.BBLConsecinta} leads us to the RSCC
$\left\{ \left([0, 1], {\mathcal{B}}_{[0, 1]}\right), \Lambda, u, P \right\}$ as follows:
\begin{equation} \label{6.1}
\left\{
\begin{array}{ll}
u : [0, 1] \times \Lambda \rightarrow [0, 1]; \quad &u(x,i):=u_{N,i}(x)=1 - \displaystyle \frac{N}{x+i},
\\
P : [0, 1] \times \Lambda \rightarrow [0, 1]; \quad  &P(x,i):= P_{N,i}(x)=\displaystyle \frac{x+N-1}{(x+i)(x+i-1)},
\\
\Lambda = \{N, N+1, \ldots \}.
\end{array}
\right.
\end{equation}
\begin{lemma} \label{lem.6.2}
$\left\{ \left([0, 1], {\mathcal{B}}_{[0, 1]}\right), \Lambda, u, P \right\}$ is an RSCC with contraction.
\end{lemma}
\begin{proof}
We have
\begin{eqnarray}
  \frac{\mathrm{d}}{\mathrm{d}x}u(x,i) &=& \frac{\mathrm{d}}{\mathrm{d}x}u_{N,i}(x) = \frac{N}{(x+i)^2}, \label{6.2} \\
  \frac{\mathrm{d}}{\mathrm{d}x}P(x,i) &=& \frac{\mathrm{d}}{\mathrm{d}x}P_{N,i}(x)=\frac{i-N}{(x+i-1)^2} - \frac{i+1-N}{(x+i)^2} \label{6.3}
\end{eqnarray}
for any $x \in [0, 1]$ and $i \in \Lambda$.
Thus,
\begin{eqnarray}
  \sup_{x \in [0, 1]}\left|\frac{\mathrm{d}}{\mathrm{d}x}u(x,i)\right| &\leq& \frac{N}{i^2}, \quad i \in \Lambda \label{6.4} \\
  \sup_{x \in [0, 1]}\left|\frac{\mathrm{d}}{\mathrm{d}x}P(x,i)\right| &<& \infty, \quad i \in \Lambda. \label{6.5}
\end{eqnarray}
Hence the requirements of Definition 3.1.15 in \cite{IG-2009} are fulfilled.
\end{proof}
Next, let  $L(W)$ denote the Banach space of all complex-valued Lipschitz continuous functions on $W$ with the following norm $\|\cdot\|_L$:
\begin{equation} \label{6.6}
\left\| f \right\|_L := \sup_{w \in W} |f(w)| + \sup_{w' \ne w''} \frac{|f(w') - f(w'')|}{|w' - w''|}, \quad f \in L(W).
\end{equation}
Whatever $t \in [0, 1]$ the Markow chain associated with RSCC (\ref{6.1}) $(s_{n,t})_{n \in {\mathbb{N}}_+}$ in (\ref{4.21})
has the transition operator $U$ in (\ref{5.2}) with the transition probability function defined as
\begin{equation}
Q(x, B) = \sum_{ \{ i \in \Lambda : u_{N,i}(x) \in B \} } P_{N,i}(x), \quad x \in [0, 1], B \in {\mathcal{B}}_{[0, 1]}. \label{6.7}
\end{equation}
\begin{lemma} \label{lem3.2}
$\left\{ \left([0, 1], {\mathcal{B}}_{[0, 1]}\right), \Lambda, u, P \right\}$ has a regular associated Markov chain.
\end{lemma}
\begin{proof}
The Markov chain corresponding to RSCC (\ref{6.1}) is regular if its transition operator $U$ is regular with respect to $L([0, 1])$.

Now, in order to prove the regularity of the associated Markov operator $U$ defined by (\ref{5.2}) with respect to $L([0, 1])$ we have, according to Theorem 3.2.13 in \cite{IG-2009}, to find an element $x^* \in I$ such that $\lim_{n \rightarrow \infty} {\rm dist}(\sigma_n(x),x^*)=0$ for all $x \in [0, 1]$.
Here $(\sigma_n)$ denotes the support of measure $Q^{n}(w, \cdot)$, where $Q^{n}$ is the $n$-step transition probability function of the Markov chain $(s_{n,t})_{n \in {\mathbb{N}}_+}$.

Fix $x \in [0, 1]$.
Let us define the sequence $(x_n)_{n \geq 0}$ in $[0, 1]$, recursively by
\begin{equation}\label{6.8}
x_0:=x, \quad x_{n+1} := 1 - \displaystyle \frac{N}{x_n + N +1}, \quad n \geq 1.
\end{equation}
It is clear that $x_n \in (0, 1)$, $n \in \mathbb{N}$, and letting $n \rightarrow \infty$ in (\ref{6.8}) we get
\begin{equation}\label{6.9}
0< x^* = \lim_{n \rightarrow \infty} x_n = \frac{\sqrt{N^2+4} - N}{2} < 1
\end{equation}
Clearly, $x_{n+1} \in \sigma_1(x_n)$ and using Lemma 3.2.14 in \cite{IG-2009} and an induction argument lead us to the conclusion that
$x_n \in \sigma_n(x)$ for $n \in \mathbb{N}_+$.
Since
${\rm dist}(\sigma_n(x),\,x^*) \leq |x_n - x^*| \rightarrow 0$ as  $n \rightarrow \infty$
we obtain that $U$ is regular.
\end{proof}
\begin{proposition} \label{prop.6.3}
$\left\{ \left([0, 1], {\mathcal{B}}_{[0, 1]}\right), \Lambda, u, P \right\}$ is uniformly ergodic.
\end{proposition}
\begin{proof}
We apply Theorem 3.4.5 in \cite{IG-2009}.
Since RSCC $\left\{ \left([0, 1], {\mathcal{B}}_{[0, 1]}\right), \Lambda, u, P \right\}$ in (\ref{6.1}) is an RSCC with contraction and is regular then the statement holds.
\end{proof}
Now, Theorem 3.1.24 \cite{IG-2009} implies that $Q^n(\cdot, \cdot)$ converges uniformly to a unique stationary probability measure $Q^{\infty}$ on $\left( [0, 1], \mathcal{B}_{[0, 1]}\right)$.
\begin{proposition} \label{prop.6.4}
The probability $Q^{\infty}$ is the invariant probability measure of the transformation $R_N$.
\end{proposition}
\begin{proof}
For $\rho_N$ in (\ref{2.5}) and $Q$ in (\ref{6.7}), and on account of the uniqueness of $Q^{\infty}$ we have to show that
\begin{equation} \label{6.010}
\int_{0}^{1} Q(x, B) \,\mathrm{d}\rho_N(x) = \rho_N(B), \quad B \in {\mathcal B}_{[0, 1]}.
\end{equation}
Since the intervals $(u, 1] \subset [0, 1]$
generate ${\mathcal B}_{[0, 1]}$, it is sufficient to show the equation (\ref{6.010})
just for $B = (u, 1]$, $0 \leq u < 1$.
Now, $u_{N,i} \in B$ iff $i \geq \left\lfloor\frac{N}{1-u}-x\right\rfloor +1$.
Thus,
\begin{eqnarray} \label{6.011}
Q(x, (u, 1]) &=& \sum_{\left\{i \in \Lambda:
u < u_{N,i}(x) \leq 1 \right\}}P_{N,i}(x) = \sum_{i \geq \left\lfloor\frac{N}{1-u}-x\right\rfloor +1} P_{N,i}(x) \nonumber \\
&=& \frac{x+N-1}{x+\left\lfloor\frac{N}{1-u}-x\right\rfloor +1}.
\end{eqnarray}
Then,
\begin{eqnarray} \label{6.0.12}
&&\int_{0}^{1} Q (x, (u, 1]) \mathrm{d}\rho_N(x)
= \frac{1}{\log \left(\frac{N}{N-1}\right)} \int_{0}^{1} \frac{\mathrm{d}x}{x+\left\lfloor\frac{N}{1-u}-x\right\rfloor} \nonumber \\
&&= \frac{1}{\log \left(\frac{N}{N-1}\right)} \left( \int_{0}^{\frac{N}{1-u}-\lfloor\frac{N}{1-u}\rfloor} \frac{\mathrm{d}x}{x+\lfloor\frac{N}{1-u}\rfloor}  + \int_{\frac{N}{1-u}-\lfloor\frac{N}{1-u}\rfloor}^{1} \ \frac{\mathrm{d}x}{x+\lfloor\frac{N}{1-u}-1\rfloor} \right) \nonumber \\
&&= \frac{1}{\log \left(\frac{N}{N-1}\right)} \log \frac{N}{u+n-1}= \rho_N \left( (u, 1]\right).
\end{eqnarray}
Hence the statement holds.
\end{proof}
Note that the Markov operator $U$ associated with RSCC (\ref{6.1}) is also given by
\begin{equation}\label{6.10}
Uf(x) = \int_{0}^{1} f(y) Q(x, \mathrm{d}y)
\end{equation}
with $Q$ as in (\ref{6.7}), which implies that
\begin{equation}\label{6.11}
U^nf(x) = \int_{0}^{1} f(y) Q^n(x, \mathrm{d}y), \quad x \in [0, 1], \ n \in \mathbb{N}_+.
\end{equation}
The operator $U^nf$ converges uniformly to a constant function $U^{\infty}f$ depending on $f$ for any $f \in L([0, 1])$ namely
\begin{equation}\label{6.12}
U^{\infty} f = \int^{1}_{0}f(y)\,Q^{\infty}(\mathrm{d}y).
\end{equation}
Furthermore, there exist two positive constants $q < 1$ and $k$ such that
\begin{equation}
\|U^n f-U^{\infty} f \|_L \leq k q^n \|f \|_L, \quad n \in \mathbb{N}_+,\, f \in L([0, 1]). \label{6.13}
\end{equation}

Now we are able to show the main result of the paper.
\begin{theorem} [A Gauss-Kuzmin-type theorem for $R_N$]
Fix an integer $N \geq 2$ and let $\left([0, 1],{\mathcal B}_{[0, 1]}, R_N \right)$ be as in Section 2.
For a non-atomic probability measure $\mu$ which has a Riemann-integrable density on $\left([0, 1], {\mathcal B}_{[0, 1]}\right)$,
the following holds:
\begin{equation} \label{6.14}
\lim_{n \rightarrow \infty}\mu (R_N^n < x) = \frac{1}{\log \left(\frac{N}{N-1}\right)}\log \frac{x+N-1}{N-1}, \quad x \in [0, 1].
\end{equation}
\end{theorem}
\begin{proof}
Let $\rho_N$ be as in (\ref{2.5}).
By (\ref{5.3}), we have
\begin{equation} \label{6.15}
\mu \left(R_N^{-n}(A)\right) = \int_{A} U^nf_0(x)\mathrm{d}\rho_N(x) \quad
\mbox{for any }n \in \mathbb{N}_+,\, A \in {\mathcal{B}}_{[0, 1]}
\end{equation}
where $f_0(x)=(x+N-1)(\mathrm{d}\mu / \mathrm{d}\lambda) (x)$ for $x \in [0, 1]$.
If $\mathrm{d}\mu / \mathrm{d}\lambda \in L([0, 1])$, then $f_0 \in L([0, 1])$
and by (\ref{6.12}) and Proposition \ref{prop.6.4} we have
\begin{eqnarray}
U^{\infty} f_0 &=& \int_{0}^{1} f_0(x)\,Q^{\infty}(\mathrm{d}x) = \int_{0}^{1} f_0(x)\,\rho_N(\mathrm{d}x) \nonumber \\
&=& \int_{0}^{1} (\mathrm{d}\mu / \mathrm{d}\lambda) (x)\,\mathrm{d}x = \frac{1}{\log \left( \frac{N}{N-1} \right) }. \label{6.16}
\end{eqnarray}

Taking into account (\ref{6.13}), there exist two constants $q<1$ and $k$ such that
\begin{equation} \label{6.16}
\left\|U^n f_0 - U^{\infty} f_0\right\|_L \leq k q^n\left\|f_0\right\|_L,
\quad n \in \mathbb{N}_+.
\end{equation}
Furthermore, consider the Banach space $C([0, 1])$ of all real-valued continuous functions on $[0, 1]$ with the norm
$\|f \| := \sup_{x \in [0, 1]}|f(x)|$.
Since $L([0, 1])$ is a dense subspace of $C([0, 1])$ we have
\begin{equation} \label{6.17}
\lim_{n \rightarrow \infty} \left\|\left(U^n - U^{\infty}\right)f\right\| = 0 \quad \mbox{for all } f \in C([0, 1]).
\end{equation}
Therefore, (\ref{6.17}) is valid for a measurable function $f_0$ which is $Q^{\infty}$-almost surely continuous, that is, for a Riemann-integrable function $f$.
Thus, we have
\begin{eqnarray*}
\lim_{n \rightarrow \infty} \mu \left(R^n_N < x\right)
&=& \lim_{n \rightarrow \infty} \int_{0}^{x} U^nf_0(u) \rho_N(\mathrm{d}u) \\
&=& \int_{0}^{x} U^{\infty} f_0(u) \rho_N(\mathrm{d}u)
= \frac{1}{\log \left(\frac{N}{N-1}\right)} \int_{0}^{x} \,\rho_N(\mathrm{d}u)\\
&=& \frac{1}{\log \left(\frac{N}{N-1}\right)}\log \frac{x+N-1}{N-1}.
\end{eqnarray*}
Hence the proof is complete.
\end{proof}



\end{document}